\documentclass[11pt]{amsart}
\usepackage{amsmath}
\usepackage{amsthm}
\usepackage{amssymb}
\usepackage{amsmath}
\usepackage{amsthm}
\usepackage{amssymb}
\usepackage{fancyhdr}
\usepackage{enumerate}
\usepackage{amscd}
\usepackage[all]{xy}
\usepackage{graphicx}

\theoremstyle{plain}

\newtheorem{lemma}{Lemma}[section]

\newtheorem{theorem}[lemma]{Theorem}
\newtheorem{remark}[lemma]{Remark}
\newtheorem{corollary}[lemma]{Corollary}

\theoremstyle{definition}
\newtheorem{example}[lemma]{Example}
\newtheorem{definition}[lemma]{Definition}

\theoremstyle{remark}
\newtheorem*{ack}{Acknowledgements}

\newcommand\rr{\mathbb{R}}

\numberwithin{equation}{section}

\begin{document}
\title[Stability of the positive mass theorem]{Stability of the positive mass theorem for graphical hypersurfaces of Euclidean space}
\author{Lan-Hsuan Huang}
\address{Department of Mathematics, University of Connecticut, Storrs, CT 06269, USA}
\email{lan-hsuan.huang@uconn.edu}
\author{Dan A. Lee}
\address{CUNY Graduate Center and Queens College}
\email{dan.lee@qc.cuny.edu}
\thanks{The first author was partially supported by the NSF through DMS-1301645 and DMS-1308837. This material is also based upon work supported by the NSF under Grant No.~0932078~000, while both authors were in residence at the Mathematical Sciences Research Institute in Berkeley, California, during the Fall 2013 program in Mathematical General Relativity.}
\date{}
\begin{abstract}
The rigidity of the positive mass theorem states that the only complete asymptotically flat manifold of nonnegative scalar curvature and \emph{zero} mass is Euclidean space. We prove a corresponding stability theorem for spaces that can be realized as graphical hypersurfaces in $\mathbb{R}^{n+1}$. Specifically, for an asymptotically flat graphical hypersurface $M^n\subset \mathbb{R}^{n+1}$ of nonnegative scalar curvature (satisfying certain technical conditions), there is a horizontal hyperplane $\Pi\subset \mathbb{R}^{n+1}$ such that the flat distance between $M$ and $\Pi$ in any ball of radius $\rho$ can be bounded purely in terms of $n$, $\rho$, and the mass of $M$. In particular, this means that if the masses of a sequence of such graphs approach zero, then the sequence weakly converges  (in the sense of currents, after a suitable vertical normalization) to a flat plane in $\mathbb{R}^{n+1}$. This result generalizes some of the earlier findings of the second author and C. Sormani \cite{Lee-Sormani:2014} and provides some evidence for a conjecture stated there.
\end{abstract}
\maketitle

\section{Introduction}
The positive mass theorem states that any complete asymptotically flat manifold of nonnegative scalar curvature has nonnegative mass. Furthermore, if the mass is zero, then the manifold must be Euclidean space. The second statement may be thought of as a rigidity theorem, and it is natural to consider the \emph{stability} of this rigidity statement. That is, if the mass is small, must the  manifold be ``close'' to Euclidean space in some sense?  Or put another way, which geometric features of the manifold can be bounded by the mass? 

One difficulty that arises in the study of stability is that the mass cannot control the geometry of a region that is separated from infinity by a minimal hypersurface. In other words, the mass cannot ``see'' the geometry behind an apparent horizon. One obvious approach to this problem is to only consider the exterior region of the manifold lying outside the outermost minimal hypersurface. However, even in the absence of minimal hypersurfaces, one can have arbitrarily deep ``gravity wells'' that make small contributions to the mass. (See \cite{Lee-Sormani:2014} for details.)  These examples show that the rigidity of the positive mass theorem cannot be stable with respect to Gromov-Hausdorff convergence, much less than any sort of smooth convergence.

Various types of stability results have appeared in the literature. In three dimensions, H.~Bray and F.~Finster~\cite{Bray-Finster:2002} used spinor methods to show that \emph{if} a sequence of smooth asymptotically flat metrics of nonnegative scalar curvature is already known to converge to a smooth limit in such a way that the mass approaches zero, and such that there are uniform bounds on the curvature and the isoperimetric constant, then the limit space must be Euclidean space. This was generalized to higher dimensional spin manifolds  by Finster and I.~Kath~\cite{Finster-Kath:2002}. Finster~\cite{Finster:2009} also obtained an upper bound on the $L^2$-norm of the curvature tensor, in terms of mass, with the exception of a set of small surface area. J.~Corvino~\cite{Corvino:2005} proved that a particular bound on the mass by the global maximum of the sectional curvature implies that the manifold is topologically trivial.  Under the assumption of conformal flatness and zero scalar curvature outside a compact set, the second author ~\cite{Lee:2009} proved that if a sequence of smooth asymptotically flat metrics of nonnegative scalar curvature has mass approaching zero, then the sequence converges smoothly to the Euclidean metric in a region outside a compact set. These various results address the stability question in the region of the manifold where the curvature tensor is small in some norm.

Another aspect of the stability problem is to understand what happens in the region of a manifold of small mass where the curvature tensor may be large.  As mentioned above, since it is known that the positive mass theorem is not stable with respect to Gromov-Hausdorff convergence, it is not clear what the optimal convergence should be. Recently, the second author and C.~Sormani~\cite{Lee-Sormani:2012, Lee-Sormani:2014} proved a stability result with respect to Sormani and S.~Wenger's ``intrinsic flat'' convergence, for the case of spherically symmetric manifolds (discussed below). 

The main purpose of this paper is to study this stability problem in a more general setting without spherical symmetry. Note that a spherically symmetric manifold can be isometrically embedded in Euclidean space as a hypersurface. This paper considers more general  asymptotically flat manifolds of nonnegative scalar curvature which can be isometrically embedded in Euclidean space as asymptotically flat graphical hypersurfaces.  In this setting, G.~Lam \cite{Lam} gave a direct proof that the ADM mass is nonnegative. The first author and D.~Wu \cite{Huang-Wu:2013} generalized it to asymptotically flat hypersurfaces which are graphical outside a compact set, and they also proved rigidity: if the mass is zero, then the hypersurface must be a hyperplane. 

As stated above, any global stability must be with respect to some sort of weak topology.
Following the work of \cite{Lee-Sormani:2014}, we choose to use the Federer-Fleming's flat topology on currents in Euclidean space. See Section \ref{flat-norm} for the definition.

\begin{theorem}\label{th:main}
Let $n\ge 5$. Let $M_i$ be a sequence of $C^{n+1}$ asymptotically flat graphs of nonnegative scalar curvature in $\rr^{n+1}$, either entire or with minimal boundary. Assume that almost every level set of $M_i$ is strictly mean-convex and outward-minimizing in the hyperplane. Normalize the height so that the level set $M_i \cap \{x^{n+1}=0\}$ has volume equal\footnote{To be precise, we normalize so that $h_0$, defined in Definition \ref{de:horizon-height}, is zero.} to $ 2\omega_{n-1} (2m)^{\frac{n-1}{n-2}}$. 

If the limit of masses of the $M_i$'s is zero, then $M_i$ weakly converges to $\{x^{n+1}=0\}$ in the sense of currents.

For $n=3$ or $4$, if we make the additional assumption that the sequence is uniformly $(r_0,\gamma,\alpha)$-asymptotically Schwarzschild (see Definition \ref{de:asymptotically-Schwarzschild}) for some choice of $(r_0, \gamma, \alpha)$ with $\alpha<0$, then we obtain the the same consequence.
\end{theorem}

See Theorems \ref{th:geq5} and \ref{th:lower-dimension} for more precise statements involving flat distance, in particular, an explicit bound on the flat distance (in a ball) between $M_i$ and $\{ x^{n+1} = 0\}$ in terms of the ADM  mass. All of the relevant definitions in Theorem \ref{th:main} appear in Section \ref{background}. 

Although the assumption that the level sets are strictly mean-convex and outward-minimizing is undesirable, it is not as restrictive as it might first appear, since our other hypotheses imply that the smooth level sets must be \emph{weakly} mean-convex (see Theorem \ref{th:mean-convexity}). Note that since the assumption is satisfied whenever the level sets are convex, it is possible to construct many examples of spaces satisfying this hypothesis: First start with a spherically symmetric asymptotically flat metric of nonnegative scalar curvature, isometrically embedded into $\rr^{n+1}$, and then perturb it slightly in any region where the scalar curvature is strictly positive.  The main reason that we require the level sets to be strictly mean-convex and outward-minimizing is the use of the Minkowski inequality for the mean curvature integral by G.~Huisken~\cite{Huisken} and by A.~Freire and F.~Schwartz~\cite{Freire-Schwartz:2014}.

The extra assumption in dimensions $3$ and $4$ is a uniformity assumption on the asymptotics of the $M_i$'s. Although this sort of assumption is a reasonable one for our stability theorem, we do not know if it is necessary. One reason why the lower dimensional case is more difficult is that asymptotically flat graphs in low dimension are not asymptotic to planes; they are unbounded at infinity.

Theorem \ref{th:main} fits together well with earlier work of the second author and Sormani \cite{Lee-Sormani:2014}.
They conjectured that the rigidity of the positive mass theorem might be stable with respect to Sormani-Wenger convergence. That is, they suggested that  a sequence of complete asymptotically flat manifolds of nonnegative scalar curvature (possibly with outermost minimal boundary) with masses approaching zero should converge to Euclidean space in the pointed Sormani-Wenger topology. The Sormani-Wenger distance between Riemannian manifolds (or more generally, between integral current spaces) is an ``intrinsic'' version of the flat distance between integral currents in Euclidean space, in essentially the same sense that Gromov-Hausdorff distance is an intrinsic version of the Hausdorff distance between subsets of Euclidean space. (See  \cite{Sormani-Wenger:2011} for details.) The second author and Sormani proved that the conjecture holds for spherically symmetric spaces, thereby establishing a proof-of-concept in a simple test case \cite{Lee-Sormani:2014}, and they proved an analogous result for the Penrose inequality~\cite{Lee-Sormani:2012} in the class of spherically symmetric spaces. A compactness result in this setting is obtained in \cite{LeFloch-Sormani:2014}. 

Although the Sormani-Wenger distance is an intrinsic version of the usual flat distance, it is important to note that they do not agree, even for submanifolds of Euclidean space.  Because of this, our results do not generalize the results of \cite{Lee-Sormani:2014}, but they are in the same spirit and may be regarded as evidence for the main conjecture in \cite{Lee-Sormani:2014}.

We summarize our approach as follows: We want to show that in any large fixed ball, a graph of small mass is close to a plane in flat distance. The mass provides a bound on a weighted total mean curvature integral for each level set $\Sigma_h$:
\begin{align*}
	m \ge \frac{1}{2(n-1) \omega_{n-1} }\int_{\Sigma_h} \frac{|D f|^2 }{1+ |Df|^2} H_{\Sigma_h} d \mathcal{H}^{n-1}.
\end{align*}
This quantity may be regarded as a quasi-local mass for level sets. Together with the Minkowski inequality, we are able to use this bound to prove a differential inequality for the volume function of the level sets, as long as the volume is not too small (Lemma \ref{le:volume-inequality-optimal}). When computing flat distance, the level sets of small volume are negligible because of the isoperimetric inequality. The differential inequality guarantees that the volumes of the remaining level sets grow as fast as they do for Schwarzschild spaces of comparable mass (see proof of Theorem \ref{th:maximum}). In particular, in dimensions greater than 4, the volume must become infinite very quickly, or in other words, all of these level sets are trapped between two planes that are a short distance apart.

In dimensions 3 and 4, we use a strong maximum principle to show that whenever a ball can be fit inside of a level set, the part of the graph outside that ball must lie beneath a corresponding Schwarzschild graph of equal mass. (See Lemma \ref{le:strong-maximum} and Figure \ref{figure:Schwarzschild}.) As this mass is small, the Schwarzschild graph is close to a plane (in the large fixed ball). The uniformity assumption that we make in low dimension allows us to see that once a level set has large enough volume, we can fit a ball of appropriate size inside (Lemma \ref{le:round-level-sets}). Finally, the differential inequality shows that just a small increase in height is enough to produce a level set that will fit one of these balls inside it.

After the paper was submitted for publication, our results  (Theorem~\ref{th:geq5} and Theorem~\ref{th:lower-dimension}), together with the results of Sormani~\cite{Sormani:2014},  were used by Sormani and the authors to settle the stability question with respect to the Sormani-Wenger topology~\cite{Huang-Lee-Sormani:2014} and thus confirmed the conjecture in~\cite{Lee-Sormani:2014} in the setting of asymptotically flat graphical hypersurfaces.

\begin{ack}
Both authors would like to thank Christina Sormani for discussions. They also thank Hugh Bray and Piotr Chru\'{s}ciel for their interest in this work.
\end{ack}

\section{Background on asymptotically flat graphs}\label{background}

\begin{definition} \label{definition:AF}
Let $f$ be a $C^1$ function defined outside a compact subset of $\mathbb{R}^n$, where $n\ge3$. We say that \emph{the graph of $f$ in $\mathbb{R}^{n+1}$ (or sometimes just $f$) is asymptotically flat}  if 
\begin{align*}
	\lim_{|x| \to \infty} f(x) & = \mbox{\rm{constant}} \mbox{ or } \pm \infty
\end{align*}
and
\begin{align*}
	\lim_{|x| \to \infty} |Df(x)| &=0. \\
\end{align*}
\end{definition}

\begin{definition} \label{de:boundary}
Let $\Omega$ be a  bounded open set in $\mathbb{R}^n$ whose complement is connected.
Let $f\in C^k(\mathbb{R}^n \setminus \overline{\Omega}) \cap C^0 (\mathbb{R}^n \setminus \Omega)$. We say that \emph{the graph of $f$ (or just $f$) is $C^k$ with a minimal boundary} if $f$ is constant on each component of $\partial \Omega$ and $|Df(x)| \to \infty$ as $x\to \partial \Omega$. A \emph{$C^k$ entire function $f$} is just a $C^k$ function defined on all of $\rr^n$.
\end{definition}

\begin{example}\label{example:Schwarzschild}
Let $M$ be a totally geodesic time-slice of the Schwarzschild spacetime of ADM mass $m$. If $m>0$, the region of $M$ outside the event horizon  can be isometrically embedded into $\mathbb{R}^{n+1}$ as the asymptotically flat graph of a smooth function defined on $\mathbb{R}^n \setminus B_{(2m)^{1/(n-2)}}(0)$, with minimal boundary, such that the boundary lies in the plane $\{ x^{n+1}=0 \}$. Explicitly, it is the graph of the function $S_m(|x|)$ where
\[
	S_m(r) =\left\{
	\begin{array}{ll}
	 \sqrt{ 8m (r - 2m)}  &\mbox{ for } n =3\\
	 \sqrt{2m} \log \left( \frac{r}{\sqrt{2m}} + \sqrt{\frac{r^2}{2m} -1} \right) &\mbox{ for } n =4\\
	S_\infty + O(r^{2-\frac{n}{2}}) &\mbox{ for } n \ge 5,
	\end{array}\right.
\]
for some constant $S_\infty$ depending on $n$ and $m$. The function $S_m$ arises from solving the ODE for a spherically symmetric graph with zero scalar curvature.
\end{example}

\begin{definition}[\cite{Lam}]
Let $f$ be an asymptotically flat $C^2$ function defined on an exterior region of $\mathbb{R}^n$. The \emph{ADM mass} of the graph of $f$ is defined by
\begin{align} \label{eq:mass} 
	m &= \frac{1}{2(n-1)  \omega_{n-1}}  \lim_{r\rightarrow \infty}\int_{|x|=r} \frac{1}{ 1 + |Df|^2 } \sum_{i,j}  (f_{ii} f_j - f_{ij} f_i)  \frac{x^j}{|x|}  \, d\mathcal{H}^{n-1},
\end{align}
where $\omega_{n-1}$ is the volume of the unit $(n-1)$-sphere.
\end{definition}
It is shown in  \cite{Lam, Huang-Wu:2013} that under the additional assumptions $|Df(x)|^2 = O_2(|x|^{-q})$  for some $q > (n-2)/2$ and 
 $|Df(x)|^2 |D^2f(x)| = o(|x|^{1-n})$, the graph of $f$ will be asymptotically flat in the usual sense, and the definition of mass above coincides with the usual definition of ADM mass.

\begin{definition}\label{de:asymptotically-Schwarzschild}
Let $\alpha< 2-\frac{n}{2}$. We say that a function $f$ is \emph{uniformly $(r_0, \gamma, \alpha)$-asymptotically Schwarzschild} if $f$ is a $C^1$ function defined on $\rr^n\setminus B_{r_0}$ and there exists a constant $\Lambda$ such that 
\[ 
	\left|f(x) - (\Lambda+ S_m(|x|)) \right| \le \gamma |x|^\alpha, 
\]
for all $|x|>r_0$, where $m$ is the mass of $f$, and $S_m$ is the Schwarzschild function described in the example above.
\end{definition}

The following identity relates the scalar curvature of an asymptotically flat graph and its mass.
\begin{theorem}[\cite{Reilly:1973}] \label{le:divergence}
Let $f\in C^2$ be defined on an open subset of $\mathbb{R}^n$. The scalar curvature of the graph of $f$ can be expressed as the divergence of a vector field as follows:
\[
    R=  \sum_j \sum_i \partial_j \left(\frac{f_{ii} f_j - f_{ij} f_i}{1+|Df|^2} \right). 
\]
\end{theorem}
Let  $\Omega_h$ be a bounded subset of $\mathbb{R}^n$ such that $\partial \Omega_h = f^{-1}(h)$, denoted by $\Sigma_h$.  Combining this theorem with the divergence theorem, and using the definition of ADM mass above, one obtains, for any regular value $h$ of $f$,
\begin{align} \label{eq:Lam-mass}
	2(n-1) \omega_{n-1} m = \int_{\mathbb{R}^n \setminus \Omega_h} R\, dx + \int_{\Sigma_h} \frac{|D f|^2 }{1+ |Df|^2} H_{\Sigma_h} d \mathcal{H}^{n-1},
\end{align}
where  $H_{ \Sigma_h}$ is the mean curvature of $\Sigma_h$ in the hyperplane $\{x^{n+1}=h\}$ with respect to inward pointing normal~\cite{Lam}. By setting $\Omega_h = \emptyset$, one immediately obtains the positive mass theorem for entire graphs.
\begin{corollary}[\cite{Lam}] Let $f$ be a $C^2$ asymptotically flat entire graph of nonnegative scalar curvature. Then its mass is nonnegative.
\end{corollary}

We also note that under the nonnegative scalar curvature hypothesis, the flux integral appearing in the definition of mass is monotone, and thus the mass always exists, though it is potentially infinite.

We recall the following theorem. 

\begin{theorem}[\cite{Huang-Wu:2013, Huang-Wu-Penrose}] \label{th:mean-convexity}
Let $M$ be a two-sided embedded $C^{n+1}$ hypersurface in $\mathbb{R}^{n+1}$ with nonnegative scalar curvature. Assume $M\setminus K$ is the union of asymptotically flat graphs where $K$ is a compact subset of $M$. Suppose that either $M$ has no boundary, or it has a minimal boundary (in the sense described in Definition~\ref{de:boundary}). Then $M$ is weakly mean-convex. Moreover, if $M$  is also minimal, then it must be a hyperplane.\footnote{This last statement, which was not stated explicitly in \cite{Huang-Wu:2013, Huang-Wu-Penrose},  is a simple consequence of \cite[Theorem 2.2]{Huang-Wu:2013}. More generally, a two-sided $C^2$ hypersurface in $\mathbb{R}^{n+1}$ with nonnegative scalar curvature and zero mean curvature must be contained in a hyperplane. }

Furthermore, for any hyperplane intersecting $M$ in a $C^2$ hypersurface $\Sigma$ in the hyperplane, we have ${\bf H}\cdot {\bf H}_{\Sigma}\ge 0$ with ${\bf H} = 0$ only if ${\bf H}_{\Sigma}=0$, where $\bf{H}$ is the mean curvature vector of $M$ in $\mathbb{R}^{n+1}$ and  $ {\bf H}_{\Sigma}$ is the mean curvature vector of $\Sigma$ in the hyperplane. As a consequence, the mean curvature scalar of $\Sigma$ inside the hyperplane has a sign.
\end{theorem}

\begin{corollary}[\cite{Huang-Wu:2013}] Let $M$ satisfy the hypotheses of Theorem~\ref{th:mean-convexity}. Then the mass of each end is nonnegative. Furthermore, if the mass of one end is zero, then $M$ is a hyperplane. 
\end{corollary}

\begin{definition}
The mean curvature vector of a hypersurface is \emph{pointing upward} if at every point the vector either points upward or is zero, and it points upward somewhere. One can define the mean curvature vector to be \emph{pointing downward} analogously. 
\end{definition}

For an asymptotically flat graph that satisfies the hypotheses of Theorem~\ref{th:mean-convexity}, the weak mean-convexity of the graph implies that we may sensibly refer to the mean curvature vector field ``pointing upward'' or ``pointing downward.'' Of course, if the mean curvature vector points downward, then simply replacing $f$ by $-f$ will yield a graph with mean curvature pointing upward. 

Our sign convention for the mean curvature vector is such that perturbations of a hypersurface in the direction of the mean curvature vector \emph{decrease} volume.
More specifically, our definition implies that the mean curvature of a unit $(n-1)$-sphere with respect to inward unit normal in $\rr^n$ has positive mean curvature $H = n-1$.  The mean curvature vectors of the Schwarzschild graphs in Example~\ref{example:Schwarzschild} are pointing upward.  In particular, a uniformly $(r_0, \gamma, \alpha)$-asymptotically Schwarzschild graph that satisfies the hypotheses of Theorem~\ref{th:mean-convexity} must have upward pointing mean curvature.

\begin{corollary} \label{co:mean-curvature-vector}
Let $f$ be a $C^{n+1}$ asymptotically flat function, either entire or with minimal boundary, whose graph has nonnegative scalar curvature and upward pointing mean curvature vector field. Let $\Sigma_h$ be a level set of a regular value $h$, then
\[
	{\bf H}_{\Sigma_h} = - H_{\Sigma_h} \frac{(Df,0)}{|Df|} \quad \mbox{and} \quad H_{\Sigma_h} \ge 0.
\]
\end{corollary}
\begin{proof}
Let ${\bf H}$ be the mean curvature vector of the graph of $f$. The upward pointing assumption implies that ${\bf H} = H\frac{(-Df, 1)}{\sqrt{1+|Df|^2}}$ and  $H\ge 0$. Then Theorem~\ref{th:mean-convexity} (that ${\bf H}\cdot {\bf H}_{\Sigma_h}\ge 0$) yields the desired result.
\end{proof}

\section{Volume estimates}
Let $f$ be an asymptotically flat function, either entire or with minimal boundary (defined on the complement of some $\Omega$ in the second case). By assumption, $f$ is constant on each component of $\partial\Omega$. Let $\bar{f}$ denote the extension of $f$ to all of $\rr^n$ such that $\bar{f}$ is constant on each component of $\bar{\Omega}$. For each $h\in\rr$, define $\Omega_h =  \{x\in \mathbb{R}^n: \bar{f}(x) < h \}$ and $ \Sigma_h =\partial^* \Omega_h$, the reduced boundary of~$\Omega_h$. (For a definition of reduced boundary, see \cite[page 72]{GMT}.) 
By Sard's Theorem, if $f$ is $C^n$ on $\rr^n\setminus\bar{\Omega}$, then the set of critical values of $f$ has zero measure, and for each regular value $h$, $\Sigma_h= \partial \Omega_h$ is just a smooth level hypersurface for $f$. We define the volume function as follows: 
\begin{align} \label{de:volume}
	V(h) = |\Sigma_h| = |\partial^* \Omega_h|,
\end{align}
where $|\Sigma_h|$ denotes the $(n-1)$-dimensional Hausdorff measure of $\Sigma_h$. Note that lower semicontinuity of perimeter implies that the function $V$ is left lower semicontinuous.

\begin{definition}
Let $E \subset \mathbb{R}^n$ be a bounded subset of finite perimeter and let $\partial^* E$ be its reduced boundary.  We say that $\partial^* E$ is \emph{outward-minimizing} if 
\[
	| \partial^* E| \le |\partial^* F| 
\]
for any bounded set $F\subset \mathbb{R}^n$ containing $E$.
\end{definition}
\begin{remark}
This property is referred to as the ``minimizing hull" property for $E$ in \cite{Huisken-Ilmanen:2001}. 
\end{remark}

\begin{lemma} \label{pr:volume-increasing}
Let $f$ be a non-constant $C^n$ asymptotically flat function, either entire or with minimal boundary, such that its graph has upward pointing mean curvature vector field. Let $h_{\textup{max}} = \lim_{|x| \to \infty}f(x)$, which is a real number or $\pm\infty$ by assumption. Then $f(x) < h_{\textup{max}}$ everywhere.

Furthermore, if $\Sigma_h$ is outward-minimizing for $h$ in a dense subset, then $V(h)$ is finite for all $h < h_{\textup{max}}$ and $V(h)$ is nondecreasing on $(-\infty, h_{\max})$.  
\end{lemma}
\begin{proof}
Since the mean curvature of the graph of $f$ points upward, the strong maximum principle for the mean curvature operator implies that $f$ cannot attain an interior local maximum unless $f$ is a constant (which it is not, by assumption). If the graph of $f$ has a minimal boundary, then  $f$ does not achieve a local maximum at the boundary; otherwise it would contradict Corollary~\ref{co:mean-curvature-vector}. Therefore $f(x) < h_{\textup{max}}$ everywhere. Thus $\Omega_h$ is a bounded subset for $h< h_{\textup{max}}$ and hence $V(h)$ is finite for any regular value $h$. By left lower semicontinuity of $V$, and density of the regular values, $V(h)$ is finite for all $h<h_{\textup{max}}$. 

We now use the outward-minimizing property to show that $V(h)$ is nondecreasing. Let $h_1<h_2 < h_{\textup{max}}$. Let $\epsilon>0$. By left lower semicontinuity of $V$ and our density assumption, there exists some  $h< h_1 < h_2$ such that $\Sigma_h$ is outward-minimizing and
\[
	V(h_1) \le V(h) + \epsilon \le V(h_2) + \epsilon.
\]
Therefore $V$ is nondecreasing.
\end{proof}

Observe that \eqref{eq:Lam-mass} and the assumption $R\ge 0$ imply
\begin{align*}
	m \ge \frac{1}{2(n-1) \omega_{n-1} }\int_{\Sigma_h} \frac{|D f|^2 }{1+ |Df|^2} H_{\Sigma_h} d \mathcal{H}^{n-1}.
\end{align*}
The goal of next two lemmas is to use the above bound to derive a differential inequality for $V(h)$.

\begin{lemma} \label{pr:volume-inequality}
Let $f \in C^{n+1}(\mathbb{R}^n\setminus \Omega)$  be an asymptotically flat function, either entire or with minimal boundary, whose graph has nonnegative scalar curvature and upward pointing mean curvature vector field. Let $h$ be a regular value of $f$.  Then for any real number $\alpha >0$, we have
\begin{align} \label{eq:volume-inequality}
	V'(h) > \alpha^{-1} \left[ \int_{\Sigma_h}  H_{\Sigma_h} - (1+\alpha^{-2}) C_n m \right],
\end{align}
 where $C_n = 2(n-1) \omega_{n-1}$.
\end{lemma}
\begin{proof}
Let $\{|Df|\ge \alpha\}$ be the set of points $x$ in $\mathbb{R}^n$ where $|Df(x)| \ge \alpha$. The set $\{ |Df| < \alpha\}$ is defined analogously.  By \eqref{eq:Lam-mass} and $R\ge 0$ we have
\begin{align} \label{equation:mass-bound}
\begin{split}
	C_n m &\ge \int_{\Sigma_h \cap \{ |Df| \ge \alpha\} } \frac{|D f|^2 }{1+ |Df|^2} H_{\Sigma_h}\\
	& \ge \frac{\alpha^2 }{1+ \alpha^2 }\int_{\Sigma_h \cap \{ |Df| \ge \alpha\} } H_{\Sigma_h}.
\end{split}
\end{align}
On the other hand, the variation of the level sets at $\Sigma_h$ for a regular value $h$ is
\[
	\frac{d}{dh} \Sigma_h(x)= \frac{Df(x)}{|Df(x)|^2}.
\] 
Since $V(h)$ is finite by Lemma~\ref{pr:volume-increasing}, we may compute $V'(h)$ using the first variation formula:
\begin{align*}
	V'(h)&=- \int_{\Sigma_h } {\bf H}_{\Sigma_h}\cdot  \frac{(Df,0)}{|Df|^2} \\
	&=\int_{\Sigma_h } \frac{H_{\Sigma_h}}{|Df|} \\
	&= \int_{\Sigma_h \cap \{ |Df| < \alpha\}}  \frac{H_{\Sigma_h}}{|Df|}  + \int_{\Sigma_h \cap \{ |Df| \ge \alpha\}}  \frac{H_{\Sigma_h}}{|Df|} \\
	& > \frac{1}{\alpha} \int_{\Sigma_h \cap \{ |Df| < \alpha\}}  H_{\Sigma_h}\\
	& = \frac{1}{\alpha} \left( \int_{\Sigma_h}  H_{\Sigma_h} - \int_{\Sigma_h \cap \{ |Df| \ge \alpha\}}  H_{\Sigma_h}\right),
\end{align*}
where we used Corollary~\ref{co:mean-curvature-vector} in the second equality. The desired result follows by substituting the second integral by \eqref{equation:mass-bound}.
\end{proof}

\begin{lemma} \label{le:volume-inequality-optimal}
Let $f$ be a $C^{n+1}$ asymptotically flat function, either entire or with minimal boundary, whose graph has nonnegative scalar curvature and upward pointing mean curvature vector field. Assume $m>0$, and let $h$ be a regular value of $f$ such that $V(h) > \omega_{n-1}(2m)^{\frac{n-1}{n-2}}$. If $\Sigma_h$ is strictly mean-convex and outward-minimizing, then 
\begin{align} \label{eq:volume2}
	V'(h)> C_n \frac{2m}{3 \sqrt{3}} \left[ \frac{1}{2m} \left( \frac{V(h)}{\omega_{n-1}}\right)^{\frac{n-2}{n-1}} - 1 \right]^{\frac{3}{2}},
\end{align}
 where $C_n = 2(n-1) \omega_{n-1}$.
\end{lemma}
\begin{proof}
Recall the following Minkowski inequality~\cite{Huisken, Freire-Schwartz:2014} for outward-minimizing $\Sigma_h$ with $H_{\Sigma_h} >0$:
\[
		\int_{\Sigma_h}  H_{\Sigma_h}  \ge \frac{C_n}{2} \left( \frac{V(h)}{\omega_{n-1}} \right)^{\frac{n-2}{n-1}}.
\]
Inserting the Minkowski inequality into \eqref{eq:volume-inequality} yields
\begin{align} \label{eq:volume}
	V'(h) > C_n \alpha^{-1} \left[ \frac{1}{2} \left(\frac{V(h)}{\omega_{n-1}} \right)^{\frac{n-2}{n-1}} - (1 + \alpha^{-2})m\right].
\end{align}
For $V(h)> \omega_{n-1}(2m)^{\frac{n-1}{n-2}}$, the right hand side of \eqref{eq:volume}, as a function of $\alpha \in (0, \infty)$, attains a global maximum at
\[
	\alpha =\sqrt{3} \left[\frac{1}{2m} \left( \frac{V(h)}{\omega_{n-1}}\right)^{\frac{n-2}{n-1}} -1\right] ^{-\frac{1}{2}}.
\]   
The desired inequality follows by inserting this choice of $\alpha$ in \eqref{eq:volume}.
\end{proof}

\begin{remark}
From our proof it is easy to see that \eqref{eq:volume2} is not optimal. For the Schwarzschild graph $h = S_m(r)$ of mass $m>0$, we can explicitly compute  
\begin{align*}
	V'(h) &= (n-1)\omega_{n-1} r^{n-2} (S_m'(r))^{-1}\\
	&=(n-1) \omega_{n-1} r^{n-2} \left[ \frac{1}{2m} r^{n-2} - 1\right]^{-\frac{1}{2}}.
\end{align*}
On the other hand, the right hand side of \eqref{eq:volume2} is
\begin{align*}
	\frac{2}{3\sqrt{3}}(n-1) \omega_{n-1}  (r^{n-2} - 2m) \left[ \frac{1}{2m} r^{n-2} - 1\right]^{-\frac{1}{2}}.
\end{align*}
\end{remark}
We now choose a height $h_0$ large enough so that the previous lemma applies for $h\geq h_0$, but still has $V(h_0)$ bounded in terms of the mass. We will argue that the graph of $f$ is close to the plane $x^{n+1} = h_0$ in the flat topology.

\begin{definition} \label{de:horizon-height}
Let $f$ be a $C^{n+1}$ asymptotically flat function, either entire or with minimal boundary, such that its graph has upward pointing mean curvature vector field, and assume $m>0$. Let $h_0$ be the height defined by
\[
	h_0 = \sup \{ h: V(h) \le 2\omega_{n-1} (2m)^{\frac{n-1}{n-2}}\}.
\]
\end{definition}

\begin{remark}
The factor of $2$ in the definition of $h_0$ is chosen for convenience. In fact, for any $\beta>1$,  one can define $h_0$ to be the supremum of $\{h: V(h) \le \beta\omega_{n-1} (2m)^{\frac{n-1}{n-2}}\}$. Then the constant $C$ in Theorem~\ref{th:maximum} depends on $\beta$, which may, however, diverge to $\infty$ as $\beta \to 1^+$.
\end{remark}

We need a comparison principle for ordinary differential inequalities. The statement is standard for $C^1$ solutions. Here we consider rough solutions.
\begin{lemma} \label{le:comparison}
Let $V:[a,b]\to \mathbb{R}$ be nondecreasing. Suppose $V' \ge F(V)$ holds almost everywhere in $[a,b]$. Suppose that $F$ is nondecreasing and continuously differentiable. Let $Y$ be a $C^2$ function satisfying
\[
	Y' = F(Y)\quad \mbox{and} \quad Y(a) \le V(a).
\]
Then $Y\le V$ on $[a,b]$.
\end{lemma}
\begin{proof}
For any $\epsilon>0$, let $Y_{\epsilon}$ be the unique $C^2$ solution to 
\[
	Y'_{\epsilon} = F(Y_{\epsilon}) \quad \mbox{and} \quad Y_{\epsilon}(a) = Y(a) - \epsilon.
\] 
Note that $Y_{\epsilon}$ varies continuously in $\epsilon$ and $\lim_{\epsilon \to 0}Y_{\epsilon} = Y$ at each point of $[a,b]$. Therefore it suffices to show that $ Y_{\epsilon}<V$ for any $\epsilon >0$.  We prove it by contradiction.  

Suppose $Y_{\epsilon}(t) \ge V(t)$ for some $\epsilon$ and for some $t\in [a,b]$.  Define
\[
	t_0 = \inf \{ t\in [a,b] : Y_{\epsilon}(t) \ge V(t)\},
\]
which exists by assumption. Since $V$ is nondecreasing and $Y$ is continuous, it follows that $t_0 > a$ and  $V(t_0)=Y_{\epsilon}(t_0)$. By the definition of $t_0$,  $Y_{\epsilon}(t)<V(t)$ on $[a, t_0)$, and therefore $F(Y_\epsilon)\le F(V)$ on $[a, t_0)$ since $F$ is nondecreasing. If $E$ is the measure zero set where $V' \ge F(V)$ fails, we have
\begin{align*}
\int_{[a,t_0)\setminus E} V'  &\ge \int_{[a,t_0)\setminus E} F(V) \\
& \ge \int_{[a,t_0)\setminus E} F(Y_\epsilon) \\
 &= \int_{[a,t_0)\setminus E} Y_\epsilon'\\
 &= \int_{[a,t_0)} Y_\epsilon' = Y_\epsilon(t_0)-Y_\epsilon(a)\\
 & = Y_\epsilon(t_0) -(Y(a)-\epsilon) \ge Y_\epsilon(t_0) - V(a) +\epsilon.
\end{align*}
 On the other hand, by the fundamental theorem of calculus for nondecreasing functions, 
 \[ V(t_0) - V(a) \ge \int_{[a,t_0)\setminus E} V' \ge Y_\epsilon(t_0) - V(a) +\epsilon.\]
So $V(t_0) > Y_\epsilon(t_0)$, which is a contradiction.
\end{proof}

\begin{theorem} \label{th:maximum}
Let $f$ be a $C^{n+1}$ asymptotically flat function, either entire or with minimal boundary, whose graph has nonnegative scalar curvature and upward pointing mean curvature vector field. Assume $m>0$, and that $\Sigma_h$ is strictly mean-convex and outward-minimizing for almost every $h$. Then the following results hold.

For $n\ge 5$, $f$ is a bounded function and there exists a constant $C$ (depending only on $n$) such that
\[
	0< \sup(f) - h_0 < C m^{\frac{1}{n-2}}.
\]

For $n= 3$ or $4$, and for any $h>h_0$, there exists an absolute constant $C$ such that
\begin{align*}
	0& \le h - h_0 \le
	  \left\{ \begin{array}{ll}
 C\sqrt{m}[V(h)]^{\frac{1}{4}} & \text{for }n=3\\
 C \sqrt{m} \log[m^{-\frac{3}{2}} V(h)]& \text{for }n=4.
 \end{array}\right.
\end{align*}

\end{theorem}
\begin{proof}
Rescale $f$ by
\[
	\tilde{f} (x) = m^{-\frac{1}{n-2}} (f(m^{\frac{1}{n-2}}x)-h_0).
\]
Then the graph of $\tilde{f}$ is asymptotically flat with nonnegative scalar curvature and mass equal to $1$.  

For $h>0$, define $\tilde{V}(h)$ to be the volume function of $\tilde{f}$ defined as in~\eqref{de:volume}, and then define $\tilde{V}$ at $0$ by
\[
	\tilde{V}(0) := \lim_{h\to 0^+} \tilde{V}(h) =   \lim_{h\to h_0^+}m^{-\frac{n-1}{n-2}}V(h)  \ge 2 \left(2^{\frac{n-1}{n-2}} \omega_{n-1}\right).
\]
By \eqref{eq:volume2}, the differential inequality
\[
		\tilde{V}'(h) > C_n \frac{2}{3 \sqrt{3}} \left[ \frac{1}{2} \left( \frac{\tilde{V}(h)}{\omega_{n-1}}\right)^{\frac{n-2}{n-1}} - 1 \right]^{\frac{3}{2}}
\] 
holds for almost every $h$ in $[0, \infty)$.

Let $Y$ be the unique smooth solution to 
\[
			Y'(h) = C_n \frac{2}{3 \sqrt{3}} \left[ \frac{1}{2} \left( \frac{{Y}(h)}{\omega_{n-1}}\right)^{\frac{n-2}{n-1}} - 1 \right]^{\frac{3}{2}} \quad \mbox{and} \quad Y(0) = 2\left(2^{\frac{n-1}{n-2}}\omega_{n-1}\right).
\]
Observe that the initial value is specifically chosen in the range such that the right hand side of the ODE is smooth in~$Y$, which explains the motivation behind the definition of $h_0$. Note that the construction of $Y$ is completely determined by $n$. We can now use our ODE comparison Lemma \ref{pr:volume-increasing} to conclude that for any $h\ge0$,
\[
	Y(h) \le \tilde{V}(h).
\] 

If $n\ge 5$, $Y(h)$ tends to infinity  at a finite height $C>0$ depending only on $n$. Therefore $\tilde{V}(h)$ must tend to infinity at a finite height $\tilde{h}_{\textup{max}} \le C$. By Lemma~\ref{pr:volume-increasing}, $\sup \tilde{f} < \tilde{h}_{\textup{max}}$ and thus $0 < \sup\tilde{f} < C$.  This implies the desired inequality for $\sup f$.

For $n=3$ or $4$, the solution $Y$ is finite for all $h\ge0$. To be more precise, we can see that for large $h$, $Y(h)$ grows like $h^4$ for $n=3$ and grows like an exponential function for $n=4$. Therefore there is a constant $C$ such that for all $h\ge0$,  $h\le C [Y(h)]^{\frac{1}{4}}$ for $n=3$, and $h\le C \log [Y(h)]$ for $n=4$. Since $Y(h) \le \tilde{V}(h)$, it follows that 
\begin{align*}
 h\le C [\tilde{V}(h)]^{\frac{1}{4}} &\text{ for }n=3\\
 h\le C \log [\tilde{V}(h)]&\text{ for }n=4.
\end{align*}
If we consider the rescaled height $\tilde{h} = m^{-\frac{1}{n-2}}(h-h_0)$, then the rescaled volume is $\tilde{V}(\tilde{h}) = m^{-\frac{n-1}{n-2}}V(h)$. The result follows.
\end{proof}

\section{Bounds in lower dimension} \label{se:lower-dimension}

In this section we prove the bounds needed for our main theorem in dimensions $3$ and $4$. This section is not used  in the proof of Theorem \ref{th:main} in dimensions $n\ge5$.

As mentioned earlier,  asymptotically flat graphs in lower dimension are not asymptotic to planes, and therefore we need a more subtle argument that requires a uniformity assumption on the asymptotics.

We start with a lemma\footnote{There is a version of this lemma for $n\ge5$, but we do not need it.} applying the strong maximum principle for the scalar curvature operator  \cite[Theorem 4.3]{Huang-Wu-Penrose} to bound $f$ outside a compact set by a vertical translation of the Schwarzschild function $S_m(|x|)$. 

\begin{lemma} \label{le:strong-maximum}
Let $n=3$ or $4$. Let $f\in C^{n+1}(\rr^n\setminus \Omega)$ be an asymptotically flat function of mass $m$, either entire or with minimal boundary, whose graph has nonnegative scalar curvature and upward pointing mean curvature vector field. 
 Let $S_m$ be the Schwarzschild function of $m>0$ defined in Example~\ref{example:Schwarzschild}, and assume that there exists some $\Lambda$ such that
 \[ f(x)-(\Lambda+S_m(|x|) )= \left\{ 
 \begin{array}{ll}
 o(|x|^{\frac{1}{2}})&\text{for }n=3\\
 o(\log|x|)&\text{for }n=4.
 \end{array}\right.
 \]   
Choose any $r_1> (2m)^{\frac{1}{n-2}}$ and any $h_1$ large enough that
 \[ B_{r_1} \subset \Omega_{h_1}.\] 
  Then for any $x\notin B_{r_1}$, 
\[
	\bar{f}(x)-h_1 \le S_m(|x|) -S_m(r_1),
\]
where $\bar{f}$ is the extension of $f$ to all of $\rr^n$ that is constant on each component of $\bar{\Omega}$.
\end{lemma}
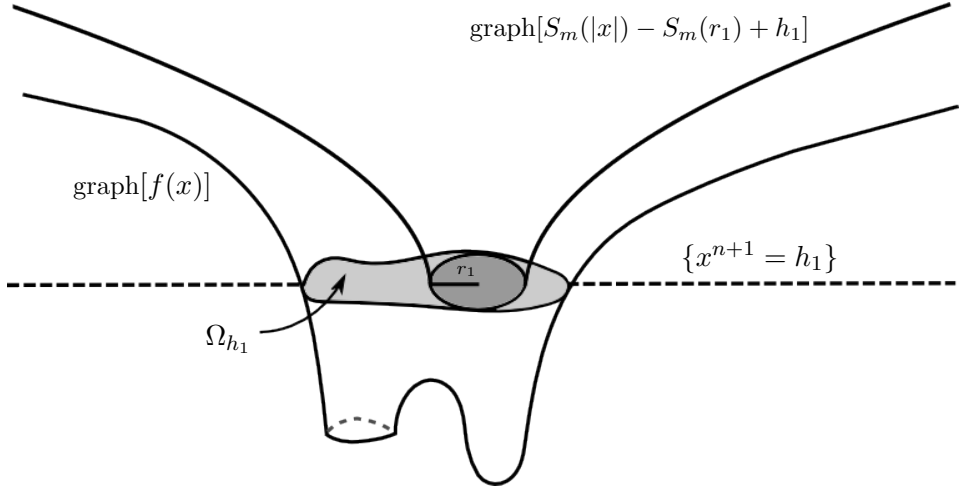
\begin{figure}[top]  
   \centering 
   \includegraphics[width=1\textwidth]{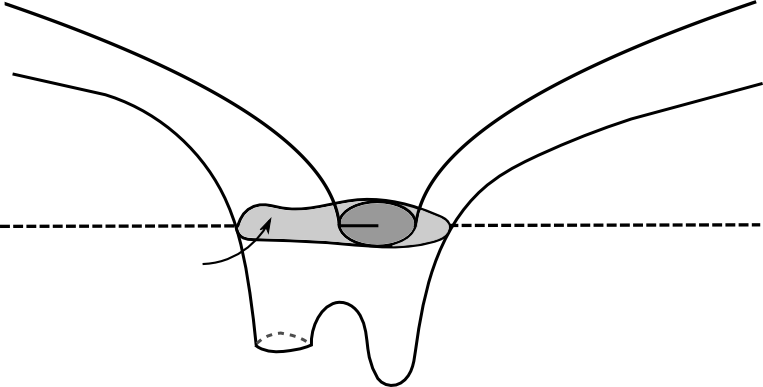}
   			\put(-335 , 110){$\small \textup{graph}[f(x)]$}
			\put(-185 , 170){\small$\textup{graph}[S_m(|x|) - S_m(r_1) + h_1]$}
			\put(-105,83) {$\small \{x^{n+1} = h_1 \}$}
			\put(-190, 80) {\tiny $r_1$}
			\put(-285, 53) {$\Omega_{h_1}$}
			\caption{The graph of $f(x)$ lies below the graph of $ S_m(|x|) -S_m(r_1) + h_1$, as shown in Lemma~\ref{le:strong-maximum}.}
			\label{figure:Schwarzschild}
\end{figure}
\begin{proof}
Suppose, to the contrary that $\bar{f}(x)-h_1>S_m(|x|) -S_m(r_1)$ for some $x\notin B_{r_1}$. Then it is also true that $\bar{f}(x)-S_{m'}(|x|)> h_1-S_{m'}(r_1)$ for some $m'$ slightly larger than $m$.  Let $a>h_1-S_{m'}(r_1)$ be the supremum of $\bar{f}(x)-S_{m'}(|x|)$ over the complement of $\Omega\cup B_{r_1}$. 

By our assumption that $B_{r_1} \subset \Omega_{h_1}$, it follows that $\bar{f}(x)\le h_1$ at $\partial B_{r_1}$, and hence  $\bar{f}(x)-S_{m'}(|x|)\le h_1-S_{m'}(r_1)<a$ at $\partial  B_{r_1}$. Therefore the supremum of $\bar{f}(x)-S_{m'}(|x|)$ cannot be achieved at $\partial B_{r_1}$. Meanwhile, our asymptotic assumption guarantees that the supremum is not achieved ``at infinity'' since $\bar{f}(x)-S_{m'}(|x|)$ approaches $-\infty$ as $|x|\to\infty$. Therefore a local maximum of $\bar{f}(x)-S_{m'}(|x|)$ is achieved at some $x'\notin \bar{B}_{r_1}$. It is clear that a local maximum cannot occur on $\Omega$. Away from $\Omega$, the graph of $f(x)$ has nonnegative scalar curvature and upward pointing mean curvature, and it touches the graph of $S_{m'}(|x|)+a$ (which has zero scalar curvature and upward pointing mean curvature) from below at the point $x'$. 
We can now invoke the strong maximum principle for the scalar curvature operator in  \cite[Theorem 4.3]{Huang-Wu-Penrose} to conclude that $f(x)= S_{m'}(|x|) +a$ on the complement of $\Omega\cup B_{r_1}$, which contradicts the fact that $f(x)$ is asymptotic to $S_m(|x|)$.
\end{proof}

Our goal in the next three technical lemmas is to obtain the radius $r_1$ for which we can bound the volume of the smallest level set that contain $B_{r_1}$ in terms of $r_1$,  under the uniformity assumption. 

\begin{lemma} \label{le:Schwarzschild-difference}
Assume $n=3$ and $\alpha<\frac{1}{2}$, or $n=4$ and $\alpha<0$. Let $1<a<b$ and $c>0$. There exists a constant $C(\gamma, \alpha, a, b, c)$ such that if 
\[
	r_1 \ge \left\{\begin{array}{ll}
		 \max\left( C(\gamma, \alpha, a, b, c) m^{- \frac{1}{1-2\alpha}}, 2m \right) &\text{for } n=3 \\
		  \max\left( C(\gamma, \alpha, a, b, c) m^{\frac{1}{2\alpha}}, \sqrt{2m} \right) &\text{for } n=4,
		 \end{array}\right.
\]
then 
\[
	\gamma (cr_1)^{\alpha} < S_m (b r_1) - S_m (a r_1).
\]	
\end{lemma}
\begin{proof}
First consider the $n=3$ case. As long as $r_1>2m$, we have
\begin{align*}
	S_m(br_1) - S_m(ar_1) &= \sqrt{8m}(\sqrt{br_1-2m} - \sqrt{ar_1-2m})\\
	&> \sqrt{8m} (\sqrt{br_1} - \sqrt{ar_1})\\
	&=\sqrt{8m} (\sqrt{b} -\sqrt{a})\sqrt{r_1},
\end{align*}
where the middle inequality follows from the fact that derivative of the square root function is positive and decreasing.

To find $r_1$ that satisfies $\gamma (cr_1)^{\alpha} <S_m (b r_1) - S_m (a r_1)$, it suffices to solve $r_1$ in the following inequality
\[
	\gamma (cr_1)^{\alpha} \le \sqrt{8m} (\sqrt{b} -\sqrt{a})\sqrt{r_1},
\]
which is equivalent to
\[
	(r_1)^{\frac{1}{2} - \alpha} \ge \left[\frac{\gamma c^{\alpha}}{\sqrt{8}(\sqrt{b} - \sqrt{a})}\right] m^{-\frac{1}{2}},
\] 
and the result follows.

For the $n=4$ case, as long as $r_1>\sqrt{2m}$, 
\begin{align*}
	S_m(br_1) - S_m(ar_1) &= \sqrt{2m}\log\left(\frac{ br_1 + \sqrt{(br_1)^2 -2m} }{ ar_1 + \sqrt{(ar_1)^2 -2m} }\right) \\
	 &> \sqrt{2m}\log\left(\frac{ br_1 }{ ar_1 }\right) \\
	 &= \sqrt{2m}\log\left(\frac{ b }{ a}\right),
\end{align*}
where the middle inequality follows from the fact that $\frac{r+\sqrt{r^2 - 2m}}{r}$ is strictly increasing for $r>2m$. Finding $r_1$ that solves  
\[
\gamma (cr_1)^{\alpha} \le \sqrt{2m}\log\left(\frac{ b }{ a}\right)
\]
yields the desired result.
\end{proof}

\begin{lemma}\label{le:round-level-sets}
Let  $n=3$ or $4$.  Let $f$ be a  uniformly $(r_0,\gamma,\alpha)$-asymptotically Schwarzschild function of mass $m>0$.  There exists a constant $C(\gamma,\alpha)$  such that if 
\[
	r_1 =\left\{\begin{array}{ll}
		 \max\left( C(\gamma, \alpha) m^{- \frac{1}{1-2\alpha}}, 2m, r_0 \right) &\text{for } n=3 \\
		  \max\left( C(\gamma, \alpha) m^{\frac{1}{2\alpha}}, \sqrt{2m}, r_0 \right) &\text{for } n=4,
		 \end{array}\right.
\]
then there exists $\epsilon>0$ such that 
\begin{align*}
 f(x)\le h_1-\epsilon &\mbox{ for } |x|={r_1}\\
f(x) \ge h_1+\epsilon&\mbox{ for }  |x|={3r_1},
\end{align*}
where $h_1=\Lambda + S_m(2r_1)$. As a consequence, for all $h \in (h_1 - \epsilon, h_1 + \epsilon)$, we have
\[
	B_{r_1} \subset \Omega_h \subset B_{3 r_1}.
\]
\end{lemma}

\begin{proof}
By assumption, there is a constant $\Lambda$ such that for all $|x|>r_0$, we have 
\[ 
	|f(x) -(\Lambda + S_m(|x|))|\le \gamma|x|^{\alpha}.
\]
Using the constant $C(\gamma, \alpha, a, b, c)$ from Lemma~\ref{le:Schwarzschild-difference}, we choose 
\[ C(\gamma,\alpha):=\max \left(  C(\gamma, \alpha, 1, 2, 1), C(\gamma,\alpha, 2, 3, 3)\right),\]
and define $r_1$ as in the statement of the lemma.

We claim that lemma holds with 
\[
	\epsilon = \min \left( S_m(2r_1) - S_m(r_1) - \gamma r_1^{\alpha}, S_m(3r_1) - S_m(2r_1) - \gamma (3 r_1)^{\alpha}\right).
\] 
First note that $\epsilon>0$ by Lemma~\ref{le:Schwarzschild-difference} (with $a=1, b=2, c=1$ for the first term and with $a=2, b=3, c=3$ for the second term). When $|x|=r_1$, we compute 
\begin{align*} 
f(x)-h_1 &= f(x) -(\Lambda + S_m(2r_1)) \\
& = [f(x) -(\Lambda + S_m(|x|))] + [S_m(r_1)-S_m(2r_1)] \\
&\le \gamma r_1^{\alpha} +  [S_m(r_1)-S_m(2r_1)]\\
&\le -\epsilon.
\end{align*}
Similarly, when $|x|=3r_1$, we have
\begin{align*} 
f(x)-h_1 &= f(x) -(\Lambda + S_m(2r_1)) \\
& = [f(x) -(\Lambda + S_m(|x|))] + [S_m(3r_1)-S_m(2r_1)] \\
&\ge -\gamma (3r_1)^{\alpha} +  [S_m(3r_1)-S_m(2r_1)]\\
&\ge\epsilon.
\end{align*}
\end{proof}

\begin{lemma}\label{le:contains-ball}
Let $n=3$ or $4$.  Let $f$ be a $C^{n+1}$ uniformly $(r_0,\gamma,\alpha)$-asymptotically Schwarzschild function of mass $m$, either entire or with minimal boundary, whose graph has nonnegative scalar curvature.   Suppose  $\Sigma_h$ is outward-minimizing for $h$ in a dense subset. Then there exists a constant $C(\gamma, \alpha)$ and a height $h_1$ such that 
\[ B_{r_1}\subset \Omega_{h_1},\]
and

	\[
	V(h_1) \le \left\{\begin{array}{ll}
		 C(\gamma, \alpha) \max\left( m^{- \frac{2}{1-2\alpha}}, m^2, r_0^2  \right) &\text{for } n=3 \\
		C(\gamma, \alpha) \max\left(  m^{\frac{3}{2\alpha}}, m^{\frac{3}{2}}, r_0^3  \right) &\text{for } n=4,
		 \end{array}\right.
		 \]
where $r_1$ is the radius defined in Lemma~\ref{le:round-level-sets}.
\end{lemma}
\begin{proof}
Clearly, there exists a regular value $h_1$ for which Lemma \ref{le:round-level-sets} implies that $\Sigma_{h_1}$ lies in the annulus $r_1 < |x| < 3r_1$ and $\Sigma_{h_1}$ is outward-minimizing. In particular, $B_{r_1}\subset \Omega_{h_1}$. Meanwhile, the outward-minimizing property implies that  $V(h_1) \le |\partial B_{3r_1}|$, and the result follows. 
\end{proof}

\begin{theorem}\label{th:lower-dim-bound}
Let $n=3$ or $4$.  Let $f$ be a $C^{n+1}$  uniformly $(r_0,\gamma,\alpha)$-asymptotically flat function defined on $\rr^n\setminus \Omega$, either entire or with minimal boundary, whose graph has nonnegative scalar curvature. Assume $m>0$, and that $\Sigma_h$ is strictly mean-convex and outward-minimizing for almost every $h$. Then there exists a constant $C$ depending only on $(r_0,\gamma,\alpha)$ such that  for $\rho>0$ and any $x\in B_\rho$, we have
 \[ \bar{f}(x)-h_0 \le 
  \left\{ \begin{array}{ll}
C (m^{\frac{-\alpha}{1-2\alpha}} + m+\sqrt{m\rho})& \text{for }n=3\\
 C \sqrt{m}( |\log m|+|\log\rho|+1) & \text{for }n=4,
 \end{array}\right.
 \]
 where $\bar{f}$ is the extension of $f$ such that $\bar{f}$ is constant on each component of $\overline{\Omega}$.
\end{theorem}
\begin{proof}
In this proof, the constant $C$ is assumed to depend on $(r_0,\gamma,\alpha)$ and may change from line to line.
We choose $h_1$ as in Lemma \ref{le:contains-ball}.  Combining Theorem \ref{th:maximum} with the volume bound on $V(h_1)$ from Lemma \ref{le:contains-ball}, we find that 
\begin{align*}
h_1-h_0  &\le \left\{ \begin{array}{ll}
 C\sqrt{m}[V(h_1)]^{\frac{1}{4}} & \text{for }n=3\\
 C \sqrt{m} \log[m^{-\frac{3}{2}} V(h_1)]& \text{for }n=4
 \end{array}\right.\\
 &\le \left\{ \begin{array}{ll}
C \max( m^{\frac{-\alpha}{1-2\alpha}}, m, \sqrt{m}) & \text{for }n=3\\
 C \sqrt{m}( |\log m|+1) & \text{for }n=4
 \end{array}\right.\\
 &\le \left\{ \begin{array}{ll}
C (m^{\frac{-\alpha}{1-2\alpha}} + m) & \text{for }n=3\\
 C \sqrt{m}( |\log m|+1) & \text{for }n=4.
 \end{array}\right.
 \end{align*} 
 From Lemma \ref{le:contains-ball}, we know $B_{r_1}\subset \Omega_{h_1}$ for the $r_1$ defined in Lemma \ref{le:round-level-sets}. In particular, $f(x)-h_0\le h_1-h_0$ for $|x|\le r_1$. For $r_1<|x|<\rho$, we can use the bound from Lemma \ref{le:strong-maximum} to compute
 \begin{align*} 
 \bar{f}(x)-h_1 &\le S_m(|x|)-S_m(r_1)\\
 & \le S_m(|x|)\\
 & < S_m(\rho)\\
 & < 
  \left\{ \begin{array}{ll}
C \sqrt{m\rho} & \text{for }n=3\\
 C \sqrt{m}( |\log \rho|+|\log m|) & \text{for }n=4.
 \end{array}\right.
 \end{align*}
 The result follows.
\end{proof}

\section{Convergence in the flat norm}\label{flat-norm}

We first recall the definition of flat distance in an open subset of $\rr^{n+1}$.

\begin{definition}
 Let $U$ be an open subset of $\mathbb{R}^{n+1}$, and let $T$ be a $k$-current in $\rr^{n+1}$. Denote by $\mathbf{M}_U$ the mass of a current in $U$. The \emph{flat norm} of $T$ in $U$ is defined by
\[
	F_U( T) = \inf\{ \mathbf{M}_U(A) + \mathbf{M}_U(B) : T = A + \partial B \mbox{ \rm{in} } U\}
\]
where the infimum is taken among $k$-currents $A$ and $(k+1)$-currents $B$ in $\rr^{n+1}$. The \emph{flat distance} between two $k$-currents $T_1, T_2$ is defined by
\[
	d_{F_U}(T_1, T_2) = F_U( T_1 - T_2).
\]
\end{definition}

Recall the height $h_0$ in Definition~\ref{de:horizon-height}.
\begin{theorem}\label{th:geq5}
Let $n\ge 5$. Let $U$ be a ball of radius $\rho$ in $\rr^{n+1}$. Let $f$ be a $C^{n+1}$ asymptotically flat function, either entire or with minimal boundary, whose graph has nonnegative scalar curvature. Assume $m>0$, and that  $\Sigma_h$ is strictly mean-convex and outward-minimizing for almost every $h$. Then 
	\[ d_{F_{U}} (\mathrm{graph} [f], \{ x^{n+1} = h_0 \}) \le c(n)[ (\rho+1)m^\frac{n}{n-2} + \rho^{n}m^\frac{1}{n-2}], \]
where the two graphs are chosen to have consistent orientations (pointing upward, say).
\end{theorem}

\begin{figure}[top] 
   \centering 
   \includegraphics[width=1\textwidth]{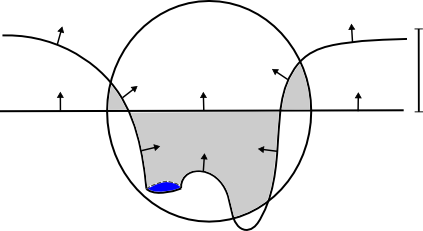}
   			\put(-120 , 175){$U$}
   			\put(-360 , 180){$\textup{graph}[f]$}
			\put(-200 , 80){$B_-$}
			\put(-115 ,110){$B_+$}
			\put(5,130) {$\le Cm^{\frac{1}{n-2}}$}
			\put(-360,85) {$\{x^{n+1} = h_0 \}$}
			\put(-223,22) {$A$}
			\caption{}
			 \label{figure}
\end{figure} 
\begin{proof}
Without loss of generality, we may assume that $f$ has upward pointing mean curvature vector. Let $M$ be the graph of $f$. We want to choose currents $A$ and $B$ such that $M- \{ x^{n+1} = h_0 \} = A+\partial B$ in $U$, where the two graphs are assumed to have upward orientation. Each component of $\partial M$ bounds a region of a horizontal plane. We define $A$ to be the sum of these regions, taken with downward orientation, so that $M - A$ is the graph of $\bar{f}$, where recall that $\bar{f}$ is the extension of $f$ such that $\bar{f}$ is constant on each component of $\overline{\Omega}$.  That is, $M-A$ is just $M$ with the boundaries ``filled in.'' We can then define $B$ to be the region of $\rr^{n+1}$ lying under $M- A$, minus the region of $\rr^{n+1}$ lying under  $\{x^{n+1} = h_0 \}$, both taken with positive orientation. 
Clearly, $(M-A) -\{ x^{n+1} = h_0 \}   =\partial B$. Note that we can think of $B=B_+ + B_-$,  
where $B_+$ is the region of $\rr^{n+1}$ with positive orientation that is below $M-A$ and above height $h_0$, and $B_-$ is the region of $\rr^{n+1}$ with negative orientation that is above $M-A$ and below height $h_0$. Since they are disjoint, $\mathbf{M}_U(B) =\mathbf{M}_U(B_+) +\mathbf{M}_U(B_-)$. See Figure \ref{figure}.

Note that since  $A$, $B_+$, and $B_-$ are submanifolds with multiplicity one, their masses in $U$ are the same as the volumes of their intersections with~$U$.
By the isoperimetric inequality combined with the Penrose inequality \cite{Lam}, we know that $\mathbf{M}_U(A)\le c(n)|\partial\Omega|^{\frac{n}{n-1}}  \le c(n) m^{\frac{n}{n-2}}$ for some $c(n)$. By Theorem~\ref{th:maximum}, it is clear that $\mathbf{M}_U(B_+)\le c(n)\rho^{n}m^{\frac{1}{n-2}}$ for some $c(n)$. Meanwhile, for almost every $h\le h_0$, we estimate the slices of $B_-\cap U$ using the isoperimetric inequality as follows:
 \begin{align*}
 \text{Vol} (B_-\cap U \cap\{x^{n+1}=h\})
& = \text{Vol} (\Omega_h\cap U)\\
 &\le  c(n) V(h)^{\frac{n}{n-1}}\\
&\le  c(n)V(h_0)^{\frac{n}{n-1}}\\
&\le c(n)  m^{\frac{n}{n-2}},\\
\end{align*}
where we used the definition of $h_0$ and the fact that $V$ is nondecreasing. Therefore
\[ \mathbf{M}_U(B_-) = \int_{-\rho}^{\rho}  \text{Vol} (B_-\cap U \cap\{x^{n+1}=h\}) \, dh\le c(n)  \rho m^{\frac{n}{n-2}}.\]
This completes the proof.
\end{proof}

\begin{theorem} \label{th:lower-dimension}
Let $n= 3$ or $4$. Let $U$ be a ball of radius $\rho$ in $\rr^{n+1}$. Let $f$ be a $C^{n+1}$  uniformly $(r_0,\gamma,\alpha)$-asymptotically flat function, either entire or with minimal boundary, whose graph has nonnegative scalar curvature. Assume $m>0$, and that  $\Sigma_h$ is strictly mean-convex and outward-minimizing for almost every $h$.  Then 
\begin{gather*}
	 d_{F_{U}} (\mathrm{graph} [f], \{ x^{n+1} = h_0 \})   \\
	 \le
	  \left\{ \begin{array}{ll}
 C(r_0,\gamma,\alpha)[(\rho+1)m^3  + \rho^3 m^{\frac{-\alpha}{1-2\alpha}}+\rho^3 m +\rho^{\frac{7}{2}}\sqrt{m}]   & \mbox{for }n=3\\
 C(r_0,\gamma,\alpha)[(\rho+1)m^2  + \rho^4 \sqrt{m}(|\log m|+|\log\rho| +1)] & \mbox{for }n=4,
 \end{array}\right. 
 \end{gather*}
where the two graphs are chosen to have consistent orientations (pointing upward, say). 
\end{theorem}
\begin{proof}
The proof is almost identical to the proof of Theorem \ref{th:geq5}. We define $A$, $B_+$, and $B_-$ as before and obtain the same estimates for $A$ and $B_-$ as follows:
\begin{align*}
	\mathbf{M}_U (A) \le c(n) m^{\frac{n}{n-2}}\quad \mbox{and} \quad
	\mathbf{M}_U(B_-) \le c(n) \rho m^{\frac{n}{n-2}}.
\end{align*}
The only difference is the estimate of $B_+$ that
\[
	M_U(B_+) \le c(n) \rho^n \sup_{B_{\rho}} (\bar{f} - h_0).
\]
In Theorem  \ref{th:geq5}, we used the fact that $\bar{f}-h_0 \le c(n)m^{\frac{1}{n-2}}$. For the lower dimensional case, we simply replace this bound by the bound on $\bar{f}-h_0$ from Theorem \ref{th:lower-dim-bound}.
\end{proof}

\begin{proof}[Proof of Theorem \ref{th:main}]
Choose a sequence $M_i$ as in the hypotheses of  Theorem \ref{th:main}, where the ``vertically normalized'' assumption is the assumption that $h_0=0$. Let $\Pi$ be the plane $\{x^{n+1}=0\}$. In order to prove that $M_i$ weakly converges to $\Pi$ in the sense of currents, let $\omega$ be a smooth compactly supported $n$-form. We must show that $(M_i-\Pi)(\omega)\to0$.

Choose $A_i$ and $B_i$ as in the proof of Theorem  \ref{th:geq5}. \[ (M_i -\Pi)(\omega) = (A_i+\partial B_i)(\omega) = A_i(\omega) + B_i(d\omega).\]
Since $\omega$ is supported in some ball $U$ of radius $\rho$, the estimates of Theorems~\ref{th:geq5} and \ref{th:lower-dimension} show that $\mathbf{M}_U(A_i)$ and $\mathbf{M}_U(B_i)$ approach zero as the masses of the $M_i$'s approach zero.  The result follows.
\end{proof}

\bibliographystyle{amsplain}
\bibliography{2014references}

\end{document}